\documentclass{amsart}
\usepackage{amssymb, amsthm, amsmath, mathtools, color}

\usepackage[unicode=true,pdfusetitle,bookmarks=true,bookmarksnumbered=false,
bookmarksopen=false,breaklinks=true,pdfborder={0 0 0},
pdfborderstyle={},backref=false,colorlinks=true]{hyperref}
\definecolor{myblue}{rgb}{0.09,0.32,0.44} %
\hypersetup{pdfborder={0 0 0},pdfborderstyle={},colorlinks=true,linkcolor=myblue,citecolor=myblue,urlcolor=blue}

\newtheorem{thm}{Theorem}[section] %
\newtheorem*{thm*}{Theorem}

\newtheorem{cor}[thm]{Corollary}
\newtheorem{defn}[thm]{Definition}

\newtheorem{exmpl}[thm]{Example}

\newtheorem{lem}[thm]{Lemma}
\newtheorem{prop}[thm]{Proposition}

\theoremstyle{remark}
\newtheorem{rem}[thm]{Remark}

\newtheorem*{rem*}{Remark}
\newtheorem*{rems*}{Remarks}

\newcommand\Cref[1]{{Corollary~\ref{#1}}}

\newcommand\Lref[1]{{Lemma~\ref{#1}}}
\newcommand\Pref[1]{{Proposition~\ref{#1}}}

\newcommand\Tref[1]{{Theorem~\ref{#1}}}
\newcommand\Sref[1]{{Section \ref{#1}}}

\newcommand{\N}{\mathbb{N}}

\newcommand{\R}{\mathbb{R}}

\newcommand{\set}[1]{\left\{#1\right\}}
\newcommand{\sub}{\subseteq}
\newcommand{\E}{\mathbb{E}}
\newcommand{\eps}{\varepsilon}
\newcommand{\floor}[1]{\left\lfloor #1 \right\rfloor}
\newcommand{\ceil}[1]{\left\lceil #1 \right\rceil}

\newcommand{\sg}[1]{\left\langle #1\right\rangle}
\newcommand{\suchthat}{\,\ifnum\currentgrouptype=16\middle\fi|\,}

\renewcommand{\Pr}{\mathbb{P}}
\newcommand{\supp}{\operatorname{supp}}

\newcommand{\Hent}{\operatorname{H}}
\newcommand{\dent}{\operatorname{h}}
\newcommand{\ind}[1]{\boldsymbol{1}_{#1}}
\newcommand{\II}{\operatorname{I}}

\newcommand{\Aut}{\operatorname{Aut}}

\newcommand{\norm}[1]{\left\| #1\right\|}
\newcommand{\En}{\operatorname{E}}
\newcommand{\Bin}{\operatorname{Bin}}

\title{Small Entropy Doubling for Random Walks and Polynomial Growth}
\author{Guy Blachar}
\address{Universit\'{e} Paris-Dauphine -- CEREMADE, Place de Lattre de Tassigny, 75016 Paris, France}
\email{guy.blachar@gmail.com}

\begin{document}

\begin{abstract}
Gromov's theorem states that a finitely generated group has polynomial growth if and only if it is virtually nilpotent. A key ingredient in its proof is the small doubling property. In this work, we study entropy analogues of this property for random walks on groups. We show that if a finitely supported symmetric random walk $R_n$ satisfies
\[
\Hent(R_{2n}) \le \Hent(R_n) + \log K
\]
at some sufficiently large scale $n$, then the underlying group is virtually nilpotent, with bounds depending on $K$ and $\mu_{\min}$.

Our approach adapts Tao's entropy Balog--Szemerédi--Gowers argument to unimodular locally compact groups, combined with structural results on approximate groups.

As applications, we obtain entropy-based criteria for polynomial growth. We also deduce an entropy gap phenomenon: if $G$ is not virtually nilpotent, then the entropy of random walks on $G$ grows faster than a universal superlogarithmic function.
\end{abstract}

\maketitle

\section{Introduction}

Let $G$ be a finitely generated group, and let $S$ be a finite symmetric generating set of $G$. The \textbf{growth function} $\gamma_{G,S}(n)$ of $G$ with respect to $S$ counts the number of elements that can be written as a product of at most $n$ elements of $S$, namely the size of the ball of radius $n$ in the Cayley graph of $G$ with respect to $S$. It is well known that the asymptotic behaviour of $\gamma_{G,S}$ does not depend on the choice of the generating set $S$, and thus one can speak simply of the \textbf{growth rate} of the group $G$.

A particularly important class is that of groups of \textbf{polynomial growth}, namely groups for which $\gamma_{G,S}(n)\le Cn^d$ for some constants $C,d>0$. Gromov's celebrated theorem \cite{Gromov81} states that finitely generated groups of polynomial growth are precisely the virtually nilpotent groups. Shalom and Tao \cite{ShalomTao10} obtained a quantitative version of this result, showing that polynomial growth at a sufficiently large scale already forces virtual nilpotence. As a consequence, there is a ``gap'' in the space of possible growth functions: there exists a superpolynomial function $g(n)$ such that every finitely generated group that is not virtually nilpotent satisfies $\gamma_{G,S}(n)\ge g(n)$.

A key ingredient in Gromov's proof is showing that if $G$ has polynomial growth, then $S$ has \textbf{small doubling} in infinitely many scales, i.e., there is a constant $K\ge 1$ such that $\left|S^{2n}\right|\le K\left|S^n\right|$ for infinitely many values of $n$. Sets with small doubling were later studied using the language of approximate groups. The latter were classified by Breuillard, Green and Tao \cite{BreuillardGreenTao12}, providing several extensions of Gromov's theorem. Works of Breuillard and Tointon \cite{BreuillardTointon16}, and of Tessera and Tointon \cite{TesseraTointon25}, established quantitative one-scale versions of the small doubling property, showing that small doubling at a sufficiently large scale implies polynomial growth (and thus virtual nilpotence).
\medskip

In this paper we study probabilistic analogues of these results, replacing volume growth by the entropy growth of random walks. Let $\mu$ be a finitely supported, symmetric, generating probability measure on $G$. A \textbf{$\mu$-random walk} on $G$ is the process $R_n=X_1\cdots X_n$, where $X_1,X_2,\dots$ are i.i.d.\ $\mu$-distributed random variables. The law of $R_n$ is the $n$-fold convolution $\mu^{*n}$ of $\mu$ with itself. Random walks on groups were studied by Kesten \cite{Kesten59}, who gave a probabilistic characterization of amenability using the spectral radius of random walks on the group, and have since been used to study other geometric properties of groups.

One quantity of interest when studying random walks is their (Shannon) entropy, namely $\Hent(R_n)=\Hent(\mu^{*n})=-\sum_g\mu^{*n}(g)\log\mu^{*n}(g)$. The entropy of $R_n$ can be thought of as the ``effective support size'' of the walk, and is therefore a natural analogue of the growth function. For instance, it follows from a work of Coulhon and Saloff-Coste \cite{CoulhonSaloffCoste93} that a group is virtually nilpotent if and only if $\Hent(R_n)\le C\log n$ for some (and hence every) finitely supported symmetric random walk $R_n$ on $G$, providing an entropy version of Gromov's theorem.

\subsection{Main results}

Our main objective in this paper is to study an entropy analogue of the small doubling property for random walks. More precisely, we investigate situations in which $\Hent(R_{2n}) \le \Hent(R_n) + \log K$ for some constant $K\ge 1$. We formulate the inequality using $\log K$ rather than $K$ in order to reflect the fact that $\Hent(R_n)$ is in a logarithmic scale compared to the growth $\left|S^n\right|$. This property was studied by Tao for probability measures on discrete abelian groups \cite{Tao10}, and was recently utilized in the proof of Marton's conjecture (the ``polynomial Freiman--Ruzsa conjecture'') in abelian groups with bounded torsion by Gowers, Green, Manners and Tao \cite{GowersGreenMannersTao25,GowersGreenMannersTao26}.

Before stating the main results, we introduce some notations. For a probability measure $\pi$ on a set $X$, we write
\[
\pi_{\min} = \min\set{\pi(x)\suchthat\pi(x)>0}.
\]
We also write $O_K(1)$ for a quantity that depends only on $K$, and $O_{K,p}(1)$ for a quantity that depends only on $K,p$.

\begin{thm}\label{thm:main}
Let $K\ge 1$ and $0<p<1$ be real numbers. Then there exists $n_0 = n_0(K,p)\in\N$ such that the following holds: Let $G$ be a finitely generated group, and let $\mu$ be a finitely supported, symmetric, symmetric probability measure on ~$G$. If $\mu_{\min}\ge p$ and
\[
  \Hent(\mu^{*2n}) \le \Hent(\mu^{*n}) + \log K
\]
for some $n\ge n_0$, then there exists a subgroup $G_0\le G$ with $[G:G_0]=O_{K,p}(1)$, and a finite normal subgroup $H\vartriangleleft G_0$ with $\left|H\right|=O_K(1)$, such that $G_0/H$ is nilpotent of rank and class at most $O_K(1)$. In particular, $G$ is virtually nilpotent.
\end{thm}

\begin{rem}
The dependence of $n_0$ and $[G:G_0]$ on $\mu_{\min}$ in the theorem is necessary. Indeed, let $G=\sg{a,b}$ be a $2$-generated group. Consider the probability measure $\mu_\eps=(1-\eps)\delta_1+\eps\nu$ on $G$, where $\nu$ is uniform on $\{a^{\pm 1},b^{\pm 1}\}$, and let $R_n^{(\eps)}$ be a $\mu_{\eps}$-random walk. Writing $M_n$ for the number of $\nu$-steps the walk $R_n^{(\eps)}$ made, we have $M_n\sim\Bin(n,\eps)$, and thus
\[
\Hent(R_n^{(\eps)}) \le \Hent(M_n) + \E[\Hent(\nu^{*M_n})] \le \Hent(M_n) + \E[M_n\log 4] = \Hent(M_n) + n\eps\log 4
\]
where the second inequality holds since $\nu$ is supported on (at most) $4$ elements. It follows that for every $n\ge 1$,
\[
\Hent(R_n^{(\eps)})\xrightarrow{\eps\to 0}0
\]
uniformly over all of the $2$-generated groups, so
\[
\Hent(R_{2n}^{(\eps)}) - \Hent(R_n^{(\eps)})\xrightarrow{\eps\to 0}0
\]
uniformly as well. We therefore see:
\begin{enumerate}
  \item Fix $K\ge 1$, and take $G=\mathrm{F}_2$ to be the free group on $2$ generators. For any given $n_0\ge 1$, by choosing $\eps$ small enough we can ensure that
      \[
      \Hent(R_{2n_0}^{(\eps)}) \le \Hent(R_{n_0}^{(\eps)}) + \log K.
      \]
      Since $\mathrm{F}_2$ is not virtually nilpotent, the conclusion of \Tref{thm:main} cannot hold, and thus $n_0$ cannot depend solely on $K$.
  \item Fix $K\ge 1$ and $n_0\ge 1$, and take $G=A_m$ the alternating group. We may again choose $\eps$ small enough so that
      \[
      \Hent(R_{2n_0}^{(\eps)}) \le \Hent(R_{n_0}^{(\eps)}) + \log K.
      \]
      The alternating groups $A_m$ are not $O_K(1)$-by-nilpotent-by-$O_K(1)$, and thus the index $[G:G_0]$ cannot be a function of $K$ alone.
\end{enumerate}
\end{rem}

The result can also be stated for vertex-transitive graphs (we assume that all graphs are simple, undirected and connected).

\begin{thm}\label{thm:main-graphs}
Let $K,D\ge 1$ be real numbers. Then there exists $n_0 = n_0(K,D)\in\N$ such that the following holds: For any locally finite vertex-transitive graph $\Gamma$ with degree at most $D$, if the simple random walk $R_n$ on $\Gamma$ satisfies
\[
\Hent(R_{2n}) \le \Hent(R_n) + \log K
\]
for some $n\ge n_0$, then $\Gamma$ has polynomial growth.
\end{thm}

It would be very interesting to find quantitative estimates on the implied constants in the above theorems. However, our techniques do not provide such estimates.

\subsection{Applications}

We present some applications of our main results to the study of entropy of random walks. We formulate the applications for random walks on groups, though they also hold for vertex-transitive graphs.

The first application provides a one-scale version of Gromov's theorem in the language of entropy:

\begin{cor}\label{cor:app-gromov}
Let $C>0$ and $0<p<1$ be real numbers. Then there exists $n_0 = n_0(C,p)\in\N$ such that the following holds: For any finitely generated group~$G$ and any finitely supported, symmetric, generating probability measure $\mu$ on $G$, if $\mu_{\min}\ge p$ and
\[
\Hent(\mu^{*n}) \le C\log n
\]
for some $n\ge n_0$, then the conclusion of \Tref{thm:main} holds.
\end{cor}

We remark that this corollary can also be deduced from a result of Tao \cite[Theorem 1.16]{Tao17}. However, it follows easily from our main results, so we include it here.

As mentioned above, the work of Shalom and Tao \cite{ShalomTao10} demonstrates the existence of a ``gap'' in the space of growth functions of groups. We prove that such a gap exists also for entropies of random walks:

\begin{cor}\label{cor:gap}
Fix $0<p<1$. Then there exists a function $f_p\colon\N\to(0,\infty)$ satisfying
\[
\lim_{n \to \infty} \frac{f_p(n)}{\log n} = \infty,
\]
such that for any non-virtually nilpotent group $G$ and any finitely supported, symmetric, generating probability measure $\mu$ on $G$ with $\mu_{\min} \ge p$, we have $\Hent(\mu^{*n}) \ge f_p(n)$.
\end{cor}

Finally, while the main result compares the random walk after $n$ and $2n$ steps, we can also provide a similar result comparing the walk after $n$ and $(1+\eps)n$ steps. A similar statement for growth is still open (see \cite[Remark 2.5]{BreuillardTointon16} and \cite[Conjecture 1.5]{TesseraTointon25}).

\begin{cor}\label{cor:app-epsilon}
Let $K\ge 1$, $0<p<1$, and $\eps>0$ be real numbers. Then there exists $n_0 = n_0(K,p,\eps)\in\N$ such that the following holds: For any finitely generated group~$G$ and any finitely supported, symmetric, generating probability measure $\mu$ on $G$, if $\mu_{\min}\ge p$ and
\[
\Hent(\mu^{*\ceil{(1+\eps)n}}) \le \Hent(\mu^{*n}) + \log K
\]
for some $n\ge n_0$, then there exists a subgroup $G_0\le G$ with $[G:G_0]=O_{K,p,\eps}(1)$, and a finite normal subgroup $H\vartriangleleft G_0$ with $\left|H\right|=O_{K,\eps}(1)$, such that $G_0/H$ is nilpotent of rank and class at most $O_{K,\eps}(1)$.
\end{cor}

\subsection*{Proof sketch and structure of the paper}

The proofs of \Tref{thm:main} and \Tref{thm:main-graphs} proceed by translating the information-theoretic assumption of small entropy doubling into a geometric structural result, using the theory of approximate groups.

The core technical engine of the paper is a version of the Balog-Szemer\'{e}di-Gowers theorem for small doubling of entropy. While Tao \cite{Tao10} previously established such a result for probability measures on discrete abelian groups, we adapt this machinery in \Pref{prop:ent-BSG} to unimodular locally compact groups. By replacing Shannon entropy with differential entropy with respect to a Haar measure, this extension allows us to simultaneously capture discrete finitely generated groups and the automorphism groups of vertex-transitive graphs.

Using this tool, we show that a measure with small entropy doubling must be nearly uniform on an approximate group in $G$ with a positive mass. We then invoke the structure theorem for approximate groups by Breuillard, Green, and Tao \cite{BreuillardGreenTao12} to deduce the existence of a subgroup $G_0$ and a finite normal subgroup $N \triangleleft G_0$ such that $G_0/N$ is nilpotent. At this stage, the random walk has constant positive measure on a coset of $G_0$. To bound the index $[G:G_0]$, we use a result of Tointon \cite[Theorem 1.11]{Tointon20}, which implies that random walks \emph{measure subgroup index uniformly}: after sufficiently many steps, the walk cannot concentrate on a subgroup of large index.

Finally, to establish \Tref{thm:main-graphs} for vertex-transitive graphs, we divide the analysis into unimodular and non-unimodular cases. The unimodular case follows the trajectory of \Tref{thm:main} via the automorphism group. For the non-unimodular case, we completely bypass the approximate group machinery. Instead, we provide a universal linear lower bound on the entropy growth by bounding the spectral radius of the Markov operator. This demonstrates that small entropy doubling is vacuously impossible on non-unimodular graphs for large $n$, completing the proof.
\medskip

In \Sref{sec:haar-entropy}, we define the notion of entropy we will use for locally compact groups and recall its basic properties. We formulate and prove our version of Balog--Szemer\'{e}di--Gowers for entropy in \Sref{sec:BSG}. We prove \Tref{thm:main} and its applications in \Sref{sec:groups}, and prove \Tref{thm:main-graphs} in \Sref{sec:graphs}.

\subsection*{Notations}

We write $\mu^{*n}$ for the $n$-fold convolution of $\mu$, and $R_n$ for the associated random walk.
For a probability measure $\pi$, we write $\pi_{\min}:=\min\{\pi(x):\pi(x)>0\}$.
We use $O(1),O_K(1),O_{K,p}(1)$ (and $\ll,\gg,\asymp$) to denote quantities bounded by a constant depending only on the indicated parameters.

\subsection*{Acknowledgements}

This work was supported by the ERC consolidator grant CUTOFF (101123174). Views and opinions expressed are however those of the authors only and do not necessarily reflect those of the European Union or the European Research Council Executive Agency. Neither the European Union nor the granting authority can be held responsible.

\section{Haar entropy}\label{sec:haar-entropy}

As mentioned above, we will formulate our main technical tool (\Pref{prop:ent-BSG}) for locally compact groups, covering both the case of discrete groups and automorphism groups of vertex-transitive graphs. To do this, we will replace Shannon entropy by differential entropy with respect to the Haar measure of the group, which we will call the Haar entropy. In this section, we introduce some notations and basic properties of the Haar entropy, which we will use throughout the paper.

\begin{defn}
Let $G$ be a locally compact group, and let $\lambda$ be a left Haar measure on $G$. Let $\mu$ be a probability measure on $G$ which is absolutely continuous with respect to $\lambda$, and write $f=\frac{d\mu}{d\lambda}$ for its density, which is Borel measurable. We define the \textbf{Haar entropy} of $\mu$ to be the differential entropy of $\mu$ with respect to $\lambda$, i.e.\
\[
\dent_\lambda(\mu) \coloneqq - \int_G f(x)\,\log f(x)\, d\lambda(x) = - \int_G \log f(x)\, d\mu(x)
\]
when the integral exists. When $X$ is a $\mu$-random variable, we write $\dent_{\lambda}(X)\coloneqq\dent_{\lambda}(\mu)$.

We use the notation $\dent_{\lambda}(X|A)$ for the conditional Haar entropy of $\mu$ conditioned on the event $A$ (with $\mu(A)>0$), $\dent_{\lambda}(X,Y)$ for the joint Haar entropy of $X,Y$, and $\dent_{\lambda}(X|Y)$ for the conditional Haar entropy of $X$ conditioned on~$Y$.
\end{defn}

\begin{exmpl}
If $G$ is a countable discrete group, then $\lambda$ is the counting measure. In this case $\dent_{\lambda}(\mu)=\Hent(\mu)$ is the standard Shannon entropy.
\end{exmpl}

\begin{rem}
We use the following properties of Haar entropy freely throughout the paper:
\begin{enumerate}
  \item $\max\set{\dent_{\lambda}(X),\dent_{\lambda}(Y)} \le \dent_{\lambda}(X,Y) \le \dent_{\lambda}(X) + \dent_{\lambda}(Y)$.
  \item For a discrete $Z$, we have $\dent_{\lambda}(X|Z) \le \dent_{\lambda}(X) \le \dent_{\lambda}(X|Z) + \Hent(Z)$.
  \item For every $g\in G$, we have $\dent_{\lambda}(gX) = \dent_{\lambda}(X)$.
  \item If $\lambda$ is also right invariant, then for every $g\in G$ we have $\dent_{\lambda}(Xg) = \dent_{\lambda}(X)$.
  \item If $X$ is supported on a measurable set $A$ and $f_X(x)\asymp \frac{1}{\lambda(A)}$ uniformly on~$A$, then $\dent_{\lambda}(X)=\log \lambda(A)+O(1)$.
\end{enumerate}
We refer the reader to \cite{CoverThomas06} for further properties on differential entropy.
\end{rem}

\begin{lem}\label{lem:dent-props}
Let $G$ be a locally compact group, and let $\lambda$ be a left Haar measure on~$G$. Let $X,Y$ be $G$-valued random variables, whose laws are absolutely continuous with respect to $\lambda$, such that $\dent_{\lambda}(X),\dent_{\lambda}(Y)<\infty$. Also, let $Z$ be a discrete random variable. Then
\[
\dent_{\lambda}(XY|Z) \le \dent_{\lambda}(X|Z) + \dent_{\lambda}(Y|Z).
\]
Furthermore, if $X,Y$ are conditionally independent relative to $Z$, then
\[
\dent_{\lambda}(XY|Z) \ge \max\set{\dent_{\lambda}(X|Z),\dent_{\lambda}(Y|Z)}.
\]
\end{lem}

\begin{proof}
The first inequality follows from
\[
\dent_{\lambda}(XY|Z) \le \dent_{\lambda}(X,Y|Z) \le \dent_{\lambda}(X|Z) + \dent_{\lambda}(Y|Z).
\]
For the second inequality, we observe that
\[
\dent_{\lambda}(XY|Z) \ge \dent_{\lambda}(XY|Y,Z) = \dent_{\lambda}(X|Y,Z) = \dent_{\lambda}(X|Z),
\]
and similarly $\dent_{\lambda}(XY|Z)\ge\dent_{\lambda}(Y|Z)$.
\end{proof}

\section{Haar entropy version of Balog--Szemer\'{e}di--Gowers}\label{sec:BSG}

In this section we prove our main technical tool -- a Balog-Szemer\'{e}di-Gowers theorem for small entropy doubling. Our proof follows the work of Tao (see \cite[Proposition 5.2]{Tao10}), who proved this proposition for discrete abelian groups.

To state the claim for locally compact groups, we recall the notion of approximate groups. Let $G$ be a unimodular locally compact group. A \textbf{$K$-approximate group} in $G$ is a symmetric, non-empty, open precompact set $H\sub G$, such that there exists a finite symmetric set $X$ of cardinality at most $K$ for which $H^2\sub XH$.

We will prove the following:

\begin{prop}\label{prop:ent-BSG}
Let $G$ be a unimodular locally compact group, and let $\lambda$ be a Haar measure on $G$. Let $\mu$ be a symmetric and compactly supported probability measure on $G$, which is absolutely continuous with respect to $\lambda$, such that $\dent_{\lambda}(\mu) < \infty$. Let $K\ge 1$ be a real number for which
\[
\dent_{\lambda}(\mu*\mu) \le \dent_{\lambda}(\mu) + \log K.
\]
Then there exists an $O_K(1)$-approximate group $H\sub G$ with $\lambda(H) \asymp_K \exp(\dent_{\lambda}(\mu))$ and a finite set $X\sub G$ of cardinality at most $O_K(1)$ such that $\mu(XH)\asymp_K 1$.
\end{prop}

We begin with the following lemma, expressing the entropy doubling difference using densities:

\begin{lem}\label{lem:ent-diff-expr}
Let $G$ be a unimodular locally compact group, and let $\lambda$ be a Haar measure on $G$. Let $X,Y$ be independent $G$-valued random variables, which are absolutely continuous with respect to $\lambda$, such that $\dent_{\lambda}(X),\dent_{\lambda}(Y)<\infty$. Then
\begin{equation}\label{eq:ent-diff-Y}
  \int_G f_X(x)\int_G f_Y(x^{-1}z)\log_+\frac{f_Y(x^{-1}z)}{f_{XY}(z)}d\lambda(z)d\lambda(x) = \dent_{\lambda}(XY) - \dent_{\lambda}(Y) + O(1)
\end{equation}
and
\begin{equation}\label{eq:ent-diff-X}
  \int_G f_Y(y)\int_G f_X(zy^{-1})\log_+\frac{f_X(zy^{-1})}{f_{XY}(z)}d\lambda(z)d\lambda(y) = \dent_{\lambda}(XY) - \dent_{\lambda}(X) + O(1),
\end{equation}
where $\log_+(t) = \max\set{\log t,0}$, and the implied constants are absolute.

\end{lem}

\begin{proof}
Write $F(t)=t\log\frac{1}{t}$ for $t\ge 0$ (with $F(0)\coloneqq 0$), so that $\dent_{\lambda}(W)=\int_G F(f_W)d\lambda$ whenever $W$ has density $f_W$ with respect to $\lambda$. Since $\lambda$ is left invariant, we have
\begin{align*}
\dent_{\lambda}(XY) - \dent_{\lambda}(Y) &= \dent_{\lambda}(XY) - \int_G f_X(x)\dent_{\lambda}(xY)d\lambda(x) \\
&= \int_G f_X(x)\int_G(F(f_{XY}(z)) - F(f_{xY}(z)))d\lambda(z)d\lambda(x) \\
&= \int_G f_X(x)\int_G(F(f_{XY}(z)) - F(f_Y(x^{-1}z)))d\lambda(z)d\lambda(x).
\end{align*}
Noting that $f_{XY}(z)=\int_G f_X(x)f_Y(x^{-1}z)d\lambda(x)$, we may insert a linear term
\[
\dent_{\lambda}(XY) - \dent_{\lambda}(Y) = \int_G f_X(x)\int_G(F(f_{XY}(z)) + F'(f_{XY}(z))(f_Y(x^{-1}z) - f_{XY}(z)) - F(f_Y(x^{-1}z)))d\lambda(z)d\lambda(x).
\]
We now use the fact that
\begin{align*}
F(b) + F'(b)(a-b) - F(a) = a\log_+\frac{a}{b} + O(a) + O(b)
\end{align*}
(where the implied constants are absolute; see \cite[equation (76)]{Tao10})
to deduce
\begin{align*}
\dent_{\lambda}(XY) - \dent_{\lambda}(Y) &= \int_G f_X(x)\int_G f_Y(x^{-1}z)\log_+\frac{f_Y(x^{-1}z)}{f_{XY}(z)}d\lambda(z)d\lambda(x) \\
& + O\left(\int_G f_X(x)\int_G f_Y(x^{-1}z)d\lambda(z)d\lambda(x)\right) \\
& + O\left(\int_G f_X(x)\int_G f_{XY}(z)d\lambda(z)d\lambda(x)\right) \\
&= \int_G f_X(x)\int_G f_Y(x^{-1}z)\log_+\frac{f_Y(x^{-1}z)}{f_{XY}(z)}d\lambda(z)d\lambda(x) + O(1)
\end{align*}
proving \eqref{eq:ent-diff-Y}. The proof of \eqref{eq:ent-diff-X} is analogous, so we omit it here.
\end{proof}

Next, we show that the small entropy doubling condition implies that the measure is close to a uniform measure on some subset, which captures a positive part of the measure. We will later see how to extract the desired approximate group from this set.

\begin{prop}\label{prop:ent-BSG-part1}
Let $G$ be a unimodular locally compact group, and let $\lambda$ be a Haar measure on $G$. Let $\mu$ be a symmetric probability measure on $G$ which is absolutely continuous with respect to $\lambda$, such that $\dent_{\lambda}(\mu) < \infty$. Let $K\ge 1$ be a real number for which
\[
\dent_{\lambda}(\mu*\mu) \le \dent_{\lambda}(\mu) + \log K.
\]
Then there exists a subset $A\sub G$ such that
\[
\lambda(A) \asymp_K \exp(\dent_{\lambda}(\mu))
\]
and
\[
f_{\mu}(x) \asymp_K \exp(-\dent_{\lambda}(\mu))
\]
uniformly for every $x\in A$
\end{prop}

\begin{proof}
We write $Z=XY$ for a product of two i.i.d.\ $\mu$-random variables $X,Y$. Fix a small number $\eps>0$, which will be chosen later and will depend only on $K$. For each $z\in G$, write
\begin{align*}
  A_z^+ & \coloneqq \set{x\in G\suchthat f_X(x) \ge e^{1/\eps}f_Z(z)}, \\
  A_z^- & \coloneqq \set{x\in G\suchthat f_X(x) \le \eps f_Z(z)}, \\
  A_z^{\circ} & \coloneqq G \setminus (A_z^+ \cup A_z^-).
\end{align*}
These sets are Borel measurable, since $f_X$ and $f_Z$ are both measurable. We will now use both inequalities of \Lref{lem:ent-diff-expr}. First, by \eqref{eq:ent-diff-Y},
\[
\int_G \int_G f_X(x)f_Y(x^{-1}z)\log_+\frac{f_Y(x^{-1}z)}{f_Z(z)}d\lambda(x)d\lambda(z) \le \log K + O(1).
\]
In particular,
\[
\int_G \int_{x:\,x^{-1}z\in A_z^+} f_X(x)f_Y(x^{-1}z)d\lambda(x)d\lambda(z) \le \eps(\log K + O(1)).
\]
We also note that
\[
\int_G \int_{x:\,x^{-1}z\in A_z^-} f_X(x)f_Y(x^{-1}z)d\lambda(x)d\lambda(z) \le \eps\int_G f_X(x)\int_G f_Z(z)d\lambda(x)d\lambda(z) = \eps
\]
by the definition of $A_z^-$.

Next, by \eqref{eq:ent-diff-X},
\[
\int_G \int_G f_Y(y)f_X(zy^{-1})\log_+\frac{f_X(zy^{-1})}{f_Z(z)}d\lambda(y)d\lambda(z) \le \log K + O(1).
\]
Substituting $x=zy^{-1}$, we get
\[
\int_G \int_G f_X(x)f_Y(x^{-1}z)\log_+\frac{f_X(x)}{f_Z(z)}d\lambda(x)d\lambda(z) \le \log K + O(1).
\]
Similarly to before,
\[
\int_G \int_{A_z^+} f_X(x)f_Y(x^{-1}z)d\lambda(x)d\lambda(z) \le \eps (\log K + O(1))
\]
and
\[
\int_G \int_{A_z^-} f_X(x)f_Y(x^{-1}z)d\lambda(x)d\lambda(z) \le \eps.
\]

Combining the above inequalities with
\[
\int_G \int_G f_X(x)f_Y(x^{-1}z)d\lambda(x)d\lambda(z) = \int_G f_Z(z)d\lambda(z) = 1,
\]
and choosing $\eps\le\frac{1}{4\log K+O(1)}$ small enough, we have
\[
\int_G \int_{x:\,x,x^{-1}z\in A_z^{\circ}} f_X(x)f_Y(x^{-1}z)d\lambda(x)d\lambda(z) \ge \frac{1}{2}.
\]
Therefore there exists $z_0\in G$ such that $f_Z(z_0)>0$ and
\begin{equation}\label{eq:z0-def}
\int_{x:\,x,x^{-1}z_0\in A_{z_0}^{\circ}} f_X(x)f_Y(x^{-1}z_0)d\lambda(x) > \frac{1}{4}f_Z(z_0).
\end{equation}
Write $A\coloneqq A_{z_0}^{\circ}$. The left hand side of \eqref{eq:z0-def} can be bounded by
\[
\int_{x:\,x,x^{-1}z_0\in A} f_X(x)f_Y(x^{-1}z_0)d\lambda(x) \le e^{2/\eps}\int_A f_Z(z_0)^2d\lambda(x) \le e^{2/\eps}f_Z(z_0)^2\lambda(A),
\]
hence
\[
\lambda(A) \gg_K \frac{1}{f_Z(z_0)}.
\]
On the other hand,
\[
1 \ge \int_A f_X(x)d\lambda(x) \ge \eps f_Z(z_0)\lambda(A),
\]
so
\[
\lambda(A) \asymp_K \frac{1}{f_Z(z_0)}.
\]
In particular, we also have
\[
f_X(x) \asymp_K \frac{1}{\lambda(A)}
\]
uniformly for all $x\in A$, and so $\mu(A)=\Pr(X\in A)\asymp_K 1$ and $\dent_{\lambda}(X|X\in A)=\log\lambda(A)+O_K(1)$.

It remains to show that
\[
\log\lambda(A) = \dent_{\lambda}(\mu) + O_K(1).
\]
Indeed, let $X_1,X_2$ be independent copies of $X$, and let $I$ denote the indicator that $X_1\in A$. Then
\begin{align*}
\dent_{\lambda}(X_1X_2) & \ge \dent_{\lambda}(X_1X_2|I) \\
&= \Pr(X_1\in A)\dent_{\lambda}(X_1X_2|X_1\in A) + \Pr(X_1\notin A)\dent_{\lambda}(X_1X_2|X_1\notin A)
\end{align*}
(if $\Pr(X_1\notin A)=0$, we interpret the last term as $0$). From \Lref{lem:dent-props},
\[
\dent_{\lambda}(X_1X_2 | X_1\in A) \ge \dent_{\lambda}(X_1 | X_1\in A) = \dent_{\lambda}(X|X\in A) = \log\lambda(A) + O_K(1)
\]
and
\[
\dent_{\lambda}(X_1X_2 | X_1\notin A) \ge \dent_{\lambda}(X_2 | X_1\notin A) = \dent_{\lambda}(X_2) = \dent_{\lambda}(X).
\]
Since by assumption $\dent_{\lambda}(X_1X_2) \le \dent_{\lambda}(X) + \log K$, we can combine all the above inequalities and deduce
\[
\Pr(X\in A)\left(\log\lambda(A) + O_K(1)\right) + \Pr(X\notin A)\dent_{\lambda}(X) \le \dent_{\lambda}(X) + \log K,
\]
which shows
\[
\log\lambda(A) \le \dent_{\lambda}(X) + O_K(1).
\]

For the reverse inequality, we note that $\dent_{\lambda}(I)\le\log 2$ since $I$ is boolean. By another use of \Lref{lem:dent-props}, we have
\[
\dent_{\lambda}(X_1X_2 | X_1\notin A) \ge \dent_{\lambda}(X_1 | X_1\notin A) = \dent_{\lambda}(X|X\notin A)
\]
and
\[
\dent_{\lambda}(X_1X_2 | X_1\in A) \ge \dent_{\lambda}(X_2 | X_1\in A) = \dent_{\lambda}(X_2) = \dent_{\lambda}(X).
\]
We then note that
\begin{align*}
  \dent_{\lambda}(X) + \log K & \ge \dent_{\lambda}(X_1X_2) \ge \dent_{\lambda}(X_1X_2 | Y) \\
  & = \Pr(X\in A)\dent_{\lambda}(X_1X_2|X_1\in A) + \Pr(X\notin A)\dent_{\lambda}(X_1X_2|X_1\notin A) \\
  & \ge \Pr(X\in A)\dent_{\lambda}(X) + \Pr(X\notin A)\dent_{\lambda}(X|X\notin A),
\end{align*}
and thus if $\Pr(X\notin A)>0$ we have $\dent_{\lambda}(X|X\notin A) \le \dent_{\lambda}(X)$.
Therefore
\begin{align*}
  \dent_{\lambda}(X) & \le \dent_{\lambda}(X_1|I) + \dent_{\lambda}(I) \\
  &= \Pr(X\in A)\dent_{\lambda}(X|X\in A) + \Pr(X\notin A)\dent_{\lambda}(X|X\notin A) + \dent_{\lambda}(I) \\
  & \le \Pr(X\in A)\log\lambda(A) + \Pr(X\notin A)\dent_{\lambda}(X) + O_K(1),
\end{align*}
which in turn implies $\dent_{\lambda}(X) \le \log\lambda(A) + O_K(1)$. This completes the proof.
\end{proof}

We recall that the \textbf{multiplicative energy} between two non-empty, open precompact subsets $A,B\sub G$ is given by
\[
\En(A,B) = \int_G [\ind{A}*\ind{B}(x)]^2 d\lambda(x)
\]
(see \cite{Tao08} for properties of the multiplicative energy). Our next goal is to show that the set $A$ of \Pref{prop:ent-BSG-part1} has large multiplicative energy.

\begin{prop}\label{prop:ent-BSG-part2}
In the setting of \Pref{prop:ent-BSG-part1}, the set $A$ of the proposition satisfies $\En(A,A)\gg_K\lambda(A)^3$.
\end{prop}

We remark that $A$ itself is precompact, since $\mu$ is compactly supported, but it might not be open. If $A$ is not open, we can use the fact that the Haar measure $\lambda$ is outer regular. This provides an open precompact set $U\supseteq A$ such that $\lambda(U)\le 1.01\lambda(A)$. The proof of the proposition can then be used with $U$ instead of $A$, and this will not interfere with the proof of \Pref{prop:ent-BSG}.

\begin{proof}
Let $X_1,X_2$ be independent $\mu$-random variables, and write $I_j$ to be the indicator of $X_j\in A$ for each $j=1,2$. By assumption,
\begin{align*}
  \dent_{\lambda}(X_1) + \log K & \ge \dent_{\lambda}(X_1X_2) \\
  & \ge \dent_{\lambda}(X_1X_2 | I_1,I_2) \\
  &= \Pr(X_1\in A)\Pr(X_2\in A)\dent_{\lambda}(X_1X_2 | X_1\in A,X_2\in A) \\
  &\phantom{=}+ \Pr(X_1\in A)\Pr(X_2\notin A)\dent_{\lambda}(X_1X_2 | X_1\in A,X_2\notin A) \\
  &\phantom{=}+ \Pr(X_1\notin A)\Pr(X_2\in A)\dent_{\lambda}(X_1X_2 | X_1\notin A,X_2\in A) \\
  &\phantom{=}+ \Pr(X_1\notin A)\Pr(X_2\notin A)\dent_{\lambda}(X_1X_2 | X_1\notin A,X_2\notin A).
\end{align*}
By \Lref{lem:dent-props}, for any two subsets $A_1,A_2\sub G$ we have
\[
\dent_{\lambda}(X_1X_2 | X_1\in A_1,X_2\in A_2) \ge \frac{1}{2}\dent_{\lambda}(X_1 | X_1\in A_1) + \frac{1}{2}\dent_{\lambda}(X_2 | X_2\in A_2),
\]
and thus we have
\begin{align*}
  \dent_{\lambda}(X_1) + \log K & \ge \Pr(X_1\in A)\Pr(X_2\in A)\left(\frac{1}{2}\dent_{\lambda}(X_1 | X_1\in A) + \frac{1}{2}\dent_{\lambda}(X_2 | X_2\in A)\right) \\
  &\phantom{=}+ \Pr(X_1\in A)\Pr(X_2\notin A)\left(\frac{1}{2}\dent_{\lambda}(X_1 | X_1\in A) + \frac{1}{2}\dent_{\lambda}(X_2 | X_2\notin A)\right) \\
  &\phantom{=}+ \Pr(X_1\notin A)\Pr(X_2\in A)\left(\frac{1}{2}\dent_{\lambda}(X_1 | X_1\notin A) + \frac{1}{2}\dent_{\lambda}(X_2 | X_2\in A)\right) \\
  &\phantom{=}+ \Pr(X_1\notin A)\Pr(X_2\notin A)\left(\frac{1}{2}\dent_{\lambda}(X_1 | X_1\notin A) + \frac{1}{2}\dent_{\lambda}(X_2 | X_2\notin A)\right) \\
  & = \Pr(X_1\in A)\dent_{\lambda}(X_1 | X_1\in A) + \Pr(X_1\notin A)\dent_{\lambda}(X_1 | X_1\notin A) \\
  & = \dent_{\lambda}(X_1 | I_1) \\
  & \ge \dent_{\lambda}(X_1) - \dent_{\lambda}(I_1) \\
  & \ge \dent_{\lambda}(X_1) - \log 2.
\end{align*}
In particular,
\[
\Pr(X_1\in A)\Pr(X_2\in A)\left(\dent_{\lambda}(X_1X_2 | X_1\in A,X_2\in A) - \frac{1}{2}\dent_{\lambda}(X_1 | X_1\in A) - \frac{1}{2}\dent_{\lambda}(X_2 | X_2\in A)\right) \le \log K + \log 2,
\]
so using $\Pr(X_i\in A)\asymp_K 1$ and $\dent_{\lambda}(X_1 | X_1\in A)=\dent_{\lambda}(X_2 | X_2\in A)$ we have
\begin{equation}\label{eq:mu'-BSG}
\dent_{\lambda}(X_1X_2 | X_1\in A,X_2\in A) - \dent_{\lambda}(X_1 | X_1\in A) \ll_K  1.
\end{equation}
Let $\mu'$ denote the law of $X_1$ conditioned on $X_1\in A$, let $X'$ be a $\mu'$-random variable, and let $f_{X'}$ denote its density with respect to $\lambda$. Also, let $Z'$ be a random variable with law $\mu'*\mu'$, with density $f_{Z'}$. Then \eqref{eq:mu'-BSG} shows that $\mu'$ satisfies the assumption of \Pref{prop:ent-BSG-part1} (with a different value of $K$ that depends only on $K$). Following the proof, we conclude that
\[
\int_G \int_{x:x,x^{-1}z\in A_z^{\circ}} f_{X'}(x)f_{X'}(x^{-1}z)d\lambda(x)d\lambda(z) \ge \frac{1}{2}
\]
(where $A_z^{\circ}$ is defined using $X'$ rather than $X$). We note that the integrand vanishes unless $x,x^{-1}z\in A_z^{\circ}$, in which case we have $f_{X'}(x),f_{X'}(x^{-1}z)\asymp_K f_{Z'}(z)$ uniformly for every $z\in G$. Therefore
\[
\int_{x:x,x^{-1}z\in A_z^{\circ}} f_{X'}(x)f_{X'}(x^{-1}z)d\lambda(x) \ll_K \lambda(A)f_{Z'}(z)^2
\]
uniformly for every $z\in G$, which implies
\[
\int_G f_{Z'}(z)^2d\lambda(z) \gg_K \frac{1}{\lambda(A)}.
\]
Since $f_{Z'} = \frac{1}{\lambda(A)^2}(\ind{A}*\ind{A})$, it follows that $\En(A,A)\gg_K \lambda(A)^3$, as required.
\end{proof}

We are ready to conclude the proof of \Pref{prop:ent-BSG}.

\begin{proof}[{Proof of \Pref{prop:ent-BSG}}]
Assume that $\dent_{\lambda}(\mu*\mu) \le \dent_{\lambda}(\mu) + \log K$. By \Pref{prop:ent-BSG-part1} and \Pref{prop:ent-BSG-part2}, there exists a subset $A\sub G$ such that $\mu(A)\asymp_K 1$, $f_{\mu}(x)\asymp_K \frac{1}{\lambda(A)}$ uniformly for every $x\in A$, and $\En(A,A)\gg_K \lambda(A)^3$. By \cite[Theorem 5.2]{Tao08}, there exist subsets $A',A''\sub A$ such that $\lambda(A'),\lambda(A'')\gg_K \lambda(A)$ and $\lambda(A'A'')\ll_K \lambda(A)$. But then by \cite[Theorem 4.6]{Tao08}, it follows that there exists an $O_K(1)$-approximate group $H\sub G$ with $\lambda(H)\ll_K \lambda(A)$ and a finite set $X\sub G$ of cardinality at most $O_K(1)$ such that $A'\sub XH$ and $A''\sub HX$. Therefore $\mu(XH)\ge\mu(A')\gg_K 1$, as required.
\end{proof}

\section{Proof for discrete groups}\label{sec:groups}

In this section, $G$ will be a discrete group. In this case $\lambda$ is the counting measure, and $G$ is unimodular.

\begin{lem}\label{lem:make-n-even}
Let $\mu$ be a finitely supported, symmetric, generating probability measure on $G$. Let $H\le G$ be a subgroup, and let $x\in G$. If $\mu^{*n}(xH)\ge\eps$ for some positive integer $n\ge 1$, then $\mu^{*2k}(H)\ge\eps$, where $2k$ is the largest even number less than or equal to $n$.
\end{lem}

\begin{proof}
We write $P$ for the Markov operator of induced random walk acting on $L^2(G/H)$. Since $\mu$ is symmetric, the operator $P$ is self-adjoint. We claim that
\begin{equation}\label{eq:ret-norm}
\mu^{*n}(xH) \le \norm{P^k\ind{H}}_2\norm{P^k\ind{xH}}_2.
\end{equation}
Indeed, we prove it depending on the parity of $n$:
\begin{itemize}
  \item If $n=2k$ is even, then
  \[
  \mu^{*n}(xH) = \left\langle \ind{H}, P^{2k}\ind{xH}\right\rangle = \left\langle P^k\ind{H}, P^k\ind{xH}\right\rangle \le \norm{P^k\ind{H}}_2\norm{P^k\ind{xH}}_2.
  \]
  \item Assume now that $n=2k+1$ is odd. In this case, we write
  \[
  \mu^{*n}(xH) = \left\langle \ind{H}, P^{2k+1}\ind{xH}\right\rangle = \left\langle P^k\ind{H}, P^{k+1}\ind{xH}\right\rangle \le \norm{P^k\ind{H}}_2\norm{P^{k+1}\ind{xH}}_2.
  \]
  Since $P$ is self-adjoint, $\norm{P}\le 1$, and thus $\norm{P^{k+1}\ind{xH}}_2\le \norm{P^k\ind{xH}}_2$. This shows that \eqref{eq:ret-norm} holds.
\end{itemize}

We now conclude the proof of the lemma. Since the Schreier graph of $G/H$ with respect to $\mu$ is transitive, we have $\norm{P^k\ind{xH}}_2=\norm{P^k\delta_{H}}_2$, and thus
\[
\mu^{*n}(xH) \le \norm{P^k\ind{H}}_2^2 = \mu^{*2k}(H)
\]
which shows that $\mu^{*2k}(H)\ge\eps$, as required.
\end{proof}

\begin{proof}[Proof of \Tref{thm:main}]
Let $\mu$ be a finitely supported, symmetric, generating probability measure on $G$ such that
\[
\Hent(\mu^{*2n}) \le \Hent(\mu^{*n}) + \log K
\]
for a sufficiently large value of $n$ (which will be chosen later depending only on $K$). By \Pref{prop:ent-BSG}, there exist an $C_1(K)$-approximate group $H\sub G$ and a finite set $X$ of cardinality at most $C_2(K)$, such that $\mu^{*n}(XH)\ge c(K)$.

We now use the structure theorem of Breuillard, Green and Tao \cite[Theorem~1.6]{BreuillardGreenTao12} to deduce that there exists a subgroup $G_0\le G$, a finite normal subgroup $N\vartriangleleft G_0$ and a finite set $X'\sub G$ of cardinality at most $C_3(K)$ such that $H\sub X'G_0$, $G_0/N$ is nilpotent and finitely generated of rank and step at most $O_K(1)$, and $G_0\sub\sg{H}$.

It follows that $\mu^{*n}(XX'G_0)\ge\mu^{*n}(XH)\ge c(K)$, so in particular there exists some $x\in XX'$ such that $\mu^{*n}(xG_0) \ge \frac{c(K)}{C_2(K)C_3(K)}\eqqcolon\eps(K)$. By \Lref{lem:make-n-even}, we have $\mu^{*2k}(G_0)\ge\eps(K)$, where $2k$ is the largest even number less than or equal to~$n$. We write $c\coloneqq(\mu*\mu)_{\min}\ge\mu_{\min}^2\ge p^2$. By \cite[Proof of Theorem 1.11]{Tointon20}, if $k\ge 1+\frac{32(1-c)^2}{c^4\eps(K)^2/9}$, then $\mu^{*2k}(L)\le\frac{2}{3}\eps$ for every subgroup $L$ with $[G:L]\ge\frac{\eps(K)}{3}$. Therefore $[G:G_0]<\frac{\eps(K)}{3}$, concluding the proof.
\end{proof}

We now turn to prove the applications for groups.

\begin{proof}[Proof of \Cref{cor:app-gromov}]
Assume that $\Hent(\mu^{*n}) \le C\log n$ for some $n\ge n_1=n_1(C,p)$ (which will be chosen later). Write $t=\floor{\frac{1}{2}\log_2 n}$ and $r=\floor{\sqrt{n}}$. Since
\[
\sum_{i=1}^t\left(\Hent(\mu^{*2^ir}) - \Hent(\mu^{*2^{i-1}r})\right) = \Hent(\mu^{*2^tr}) - \Hent(\mu^{*r}) \le \Hent(\mu^{*n}) \le C\log n,
\]
there exists some $1\le i\le t$ such that
\[
\Hent(\mu^{*2^ir}) - \Hent(\mu^{*2^{i-1}r}) \le \frac{C}{t}\log n \le \frac{C\log n}{\frac{1}{2}\log_2 n-1} = \frac{2C\log 2}{1-\frac{2}{\log_2 n}} \le 4C\log 2
\]
where the last inequality holds when $n\ge 16$. We therefore take $K=4C\log 2$, and let $n_0=n_0(K,p)$ be the value of $n_0$ from \Tref{thm:main}. Then by choosing $n_1 \ge \max\set{1+n_0^2,16}$, the corollary follows from \Tref{thm:main}.
\end{proof}

\begin{proof}[Proof of \Cref{cor:gap}]
Fix $0<p<1$. For each $j\ge 1$, let $m_j$ be the value of $n_0(j,p)$ from \Cref{cor:app-gromov}. We may assume that $m_0\coloneqq 0\le m_1\le m_2\le\cdots$. We define the function $f_p\colon\N\to\R$ by $f_p(n)=j\log n$ for the unique value of $j$ such that $m_j\le n<m_{j+1}$. It is clear that $\lim_{n\to\infty}\frac{f_p(n)}{\log n}=\infty$. If $G$ is a non-virtually nilpotent group, and $\mu$ is a finitely supported, symmetric, generating probability measure on $G$ with $\mu_{\min}\ge p$, then \Cref{cor:app-gromov} forces $\Hent(\mu^{*n})\ge j\log n=f_p(n)$ whenever $m_j\le n<m_{j+1}$, as required.
\end{proof}

\begin{proof}[Proof of \Cref{cor:app-epsilon}]
Suppose that
\[
\Hent(\mu^{*\ceil{(1+\eps)n}}) \le \Hent(\mu^{*n}) + \log K.
\]
By \cite{KaimanovichVershik83}, the function $n\mapsto\Hent(\mu^{*n})$ is concave. Therefore
\[
\Hent(\mu^{*2n}) - \Hent(\mu^{*n}) \le \frac{1}{\eps}\left(\Hent(\mu^{*\ceil{(1+\eps)n}}) - \Hent(\mu^{*n})\right) \le \frac{1}{\eps}\log K.
\]
The corollary now follows from \Tref{thm:main}.
\end{proof}

\section{Proof for vertex-transitive graphs}\label{sec:graphs}

In this section, we prove \Tref{thm:main-graphs}. We fix some notations for this section. Let $\Gamma=(V,E)$ be a connected locally finite vertex-transitive graph of degree $d$ with root $o$. We write $R_n$ for the simple random walk on $\Gamma$ starting from $o$. Let $G=\Aut(\Gamma)$ be the automorphism group of $\Gamma$, which is a locally compact group. For every $v\in V$, we write
\[
G_v = \set{g\in G\suchthat gv=v}
\]
for the stabilizer of $v$ in $G$, and for every $v,w\in V$ we write
\[
G_{v,w} = \set{g\in G\suchthat gv=w}.
\]

We recall that the graph $\Gamma$ is called \textbf{unimodular} if $\left|G_vw\right| = \left|G_wv\right|$ for every $v,w\in V$. In this case $G$ is also unimodular, so any left Haar measure on $G$ is also a right Haar measure. We will prove \Tref{thm:main-graphs} by dividing it to two cases, depending on whether $\Gamma$ is unimodular or not.

\subsection{The unimodular case}

We assume that $\Gamma$ is unimodular, and let $\lambda$ be a Haar measure on~$G$, normalized so that $\lambda(G_o)=1$. Taking $g_{v,o}\in G_{v,o}$ and $g_{o,w}\in G_{o,w}$, we have that $G_{v,w}=g_{o,w}G_og_{v,o}$, so $\lambda(G_{v,w})=1$ for every $v,w\in V$.

We define the probability measure $\mu=\frac{1}{d}\sum_{v\sim o}\lambda|_{G_{o,v}}$ on $G$.

\begin{prop}\label{prop:graphs-mu-props}
Suppose that $\Gamma$ is unimodular. Then:
\begin{enumerate}
  \item $\mu$ is symmetric and compactly supported.
  \item $\dent_{\lambda}(\mu^{*n})=\Hent(R_n)$ for every $n\ge 1$.
\end{enumerate}
\end{prop}

\begin{proof}
\begin{enumerate}
  \item Let $S = \supp(\mu) = \bigcup_{v\sim o}G_{o,v} = \set{g\in G\suchthat go\sim o}$, which is compact since each $G_{o,v}$ is compact, and symmetric since $go\sim o\iff g^{-1}o\sim o$. Then $\mu(A)=\frac{1}{d}\lambda(A\cap S)$ for any measurable $A\sub G$. Hence $\mu(A^{-1}) = \frac{1}{d}\lambda(A^{-1}\cap S) = \frac{1}{d}\lambda((A\cap S)^{-1}) = \frac{1}{d}\lambda(A\cap S)=\mu(A)$, where the second equality used the symmetry of $S$, and the third used the symmetry of $\lambda$ (which follows from the unimodularity assumption).
  \item By definition, the measure $\mu$ is right $G_o$-invariant. Therefore, the density of $\mu^{*n}$ with respect to $\lambda$ is constant on each left coset $G_{o,v}$ of $G_o$, and thus it must be equal to $f = \sum_{v\in V}\Pr(R_n=v)\ind{G_{o,v}}$, hence
      \begin{align*}
        \dent_{\lambda}(\mu^{*n}) &= -\int_G f(x)\log f(x)d\lambda(x) \\
        &= -\sum_{v\in V}\Pr(R_n=v)\log\Pr(R_n=v)\lambda(G_{o,v}) = \Hent(R_n),
      \end{align*}
      as required. \qedhere
\end{enumerate}
\end{proof}

We also need the following lemma, showing that the index of open subgroups of a unimodular locally compact group can be identified using random walks.

\begin{lem}\label{lem:tointon-lc}
Let $G$ be a unimodular locally compact group, and let $\mu$ be a finitely supported, symmetric probability measure on $G$. Let $G_0 \le G$ be an open subgroup. If $\mu^{*2k}(G_0) \ge \eps$ for some integer $k \ge 1$, then the index $[G:G_0]$ is bounded by the same uniform constant as in the discrete case.
\end{lem}

\begin{proof}
Since $G_0$ is an open subgroup of $G$, the left coset space $X = G/G_0$ is discrete. The measure $\mu$ induces a random walk on $X$ with transition probabilities $p(xG_0, yG_0) = \mu(x^{-1}y G_0)$. Because $G$ is unimodular, the counting measure on $X$ is $G$-invariant. Furthermore, since $\mu$ is symmetric, the transition operator of the induced random walk is self-adjoint on $L^2(X)$, making the walk reversible. The uniform bounds on subgroup index established by Tointon \cite[Theorem 1.11]{Tointon20} rely only on the discreteness of the state space and the reversibility of the Markov chain. As our induced random walk on $X$ satisfies both conditions, Tointon's arguments apply verbatim, yielding the required uniform bound on $|X| = [G:G_0]$.
\end{proof}

We can now prove \Tref{thm:main-graphs} for unimodular graphs.

\begin{proof}[Proof of \Tref{thm:main-graphs}, unimodular case]
Assume that $\Gamma$ is unimodular. By assumption,
\[
\Hent(R_{2n}) \le \Hent(R_n) + \log K
\]
for a sufficiently large value of $n$, which will be chosen later. By \Pref{prop:graphs-mu-props}, we have
\[
\dent_{\lambda}(\mu^{*2n}) \le \dent_{\lambda}(\mu^{*n}) + \log K.
\]

We repeat the beginning of the proof of \Tref{thm:main}. This yields a subgroup $G_0\le G$, a finite normal subgroup $N\vartriangleleft G_0$, and an element $x\in G$, such that $\mu^{*n}(xG_0)\ge\eps(K)>0$. We also note that $\lambda(G_0)>0$, hence $G_0$ is open by Steinhaus' theorem. \Lref{lem:make-n-even} shows that we may assume that $\mu^{*2k}(G_0)\ge\eps(K)$, where~p$2k$ is the largest even number less than or equal to $n$. We now use \Lref{lem:tointon-lc}, so we must have $[G:G_0]\le\frac{2}{\eps(K)}$, showing that $G$ is virtually nilpotent. Therefore $\Gamma$ has polynomial growth.
\end{proof}

\subsection{The non-unimodular case}

When $\Gamma$ is not unimodular, then $G$ is not amenable by \cite{SoardiWoess90}, and thus the entropy $\Hent(R_n)$ grows linearly with $n$. In this subsection, we give a universal linear lower bound on this entropy which depends only on the degree $d$. To do this, we write
\[
h = \lim_{n\to\infty}\frac{1}{n}\Hent(R_n)
\]
for the asymptotic entropy of $R_n$.

\begin{prop}\label{prop:non-uni-ent-bound}
Suppose that $\Gamma$ is not unimodular. Then
\[
h \ge -\frac{1}{2}\log\left(1-\frac{1}{d^2}\right) \ge \frac{1}{2d^2}.
\]
\end{prop}

\begin{proof}
Let $P$ denote the Markov operator of the random walk on $\Gamma$. Then $P=\frac{1}{d}A$, where $A$ is the adjacency matrix of $\Gamma$. Since $P$ is reversible, the spectral radius of~$P$ is $\rho(P)=\norm{P}_{2\to 2} = \frac{1}{d}\norm{A}_2$.

The action of $G_o$ on the $d$ neighbors of $o$ divides them into orbits $O_1,\dots,O_t$. For each $1\le i\le t$, choose a representative $v_i\in O_i$. For each $1\le i\le t$, consider the directed subgraph $\Gamma_i$ induced from all edges in $\Gamma$ of the form $(go,gv_i)$ for $g\in G$, and let $A_i$ denote its adjacency matrix. Then $A=\sum_{i=1}^t A_i$. Each vertex in $\Gamma_i$ has out-degree $a_i\coloneqq\left|G_ov_i\right|$, and in-degree $b_i\coloneqq\left|G_{v_i}o\right|$. We note that $\sum_{i=1}^t a_i = \sum_{i=1}^t b_i = d$. Furthermore, since $\Gamma$ is not unimodular, there exists some $1\le j\le t$ such that $a_j\ne b_j$.

By Schur's test, for each $1\le i\le t$ we have $\norm{A_i}_2\le\sqrt{a_ib_i}$, so by the triangle inequality we have
\[
\rho(P) = \frac{1}{d}\norm{A}_2 \le \frac{1}{d}\sum_{i=1}^t\norm{A_i}_2 \le \frac{1}{d}\sum_{i=1}^t\sqrt{a_ib_i}.
\]
We claim that $\rho(P)\le\sqrt{1-\frac{1}{d^2}}$. This will follow from the following lemma:

\begin{lem}
Let $a_1,\dots,a_t,b_1,\dots,b_t$ be positive integers such that $\sum_{i=1}^t a_i = \sum_{i=1}^t b_i = d$, and $\vec{a}\ne\vec{b}$. Then
\[
\sum_{i=1}^t\sqrt{a_ib_i} \le \sqrt{d^2-1}.
\]
\end{lem}

\begin{proof}
We write $I_{\le}=\set{1\le i\le t\suchthat a_i\le b_i}$ and $I_{>}=\set{1\le i\le t\suchthat a_i>b_i}$. By Cauchy-Schwarz, $\sqrt{x_1y_1} + \sqrt{x_2y_2} \le \sqrt{(x_1+x_2)(y_1+y_2)}$ for all $x_i,y_i\ge 0$. Therefore, if both $i,j\in I_{\le}$ (or $i,j\in I_{>}$), replacing $(a_i,b_i)$ and $(a_j,b_j)$ with $(a_i+a_j,b_i+b_j)$ does not decrease $\sum_{i=1}^t\sqrt{a_ib_i}$. Writing $a=\sum_{i\in I_{\le}}a_i$ and $b=\sum_{i\in I_{>}}b_i$, we deduce that
\[
\sum_{i=1}^t\sqrt{a_ib_i} \le \sqrt{ab} + \sqrt{(d-a)(d-b)}
\]
and $a<b$. Substituting $u=\frac{a+b}{2}$ and $r=\frac{b-a}{2}$, we have
\[
\sqrt{ab} + \sqrt{(d-a)(d-b)} = \sqrt{u^2-r^2} + \sqrt{(d-u)^2-r^2}.
\]
For a given $r\ge 0$, the function $x\mapsto\sqrt{x^2-r^2}$ is concave. Hence
\[
\sqrt{u^2-r^2} + \sqrt{(d-u)^2-r^2} \le 2\sqrt{\frac{d^2}{4}-r^2} = \sqrt{d^2-4r^2}.
\]
Also, since $b\ge a+1$ we have $r\ge\frac{1}{2}$, hence $\sqrt{d^2-4r^2}\le\sqrt{d^2-1}$. Together this concludes the proof of the lemma.
\end{proof}
We can now finish the proof of the proposition. We note that
\[
\frac{1}{n}\Hent(R_n) \ge -\frac{1}{n}\log\max_{v\in V}\Pr(R_n=v) \ge -\frac{1}{2n}\log\sum_{v\in V}\Pr(R_n=v)^2 = -\log \Pr(R_{2n}=o)^{1/2n}.
\]
Taking the limit of both sides, we deduce
\[
h = \lim_{n\to\infty}\frac{1}{n}\Hent(R_n) \ge -\log \lim_{n\to\infty}\Pr(R_{2n}=o)^{1/2n} = -\log\rho(P) \ge -\frac{1}{2}\log\left(1-\frac{1}{d^2}\right) \ge \frac{1}{2d^2},
\]
where the last inequality follows from $\log(1+x)\le x$.
\end{proof}

\begin{cor}\label{cor:non-uni-diffent-bound}
When $\Gamma$ is not unimodular, we have
\[
\Hent(R_{2n}) - \Hent(R_n) \ge \frac{n}{2d^2}.
\]
for every $n\ge 1$.
\end{cor}

\begin{proof}
We first note that $\Hent(R_n)-\Hent(R_{n-1}) = \Hent(R_n) - \Hent(R_n|R_1) = \II(R_n;R_1)$, where the first equality follows from the vertex-transitivity of $\Gamma$. By the data processing inequality, we have $\Hent(R_n) - \Hent(R_{n-1}) \ge \Hent(R_{n+1}) - \Hent(R_n)\ge 0$ for every $n\ge 1$. Therefore the sequence $\Hent(R_n) - \Hent(R_{n-1})$ converges, and its limit must be~$h$ since $\frac{1}{n}\Hent(R_n) = \sum_{i=1}^n(\Hent(R_i)-\Hent(R_{i-1})$. Finally,
\[
\Hent(R_{2n}) - \Hent(R_n) = \sum_{i=n+1}^{2n}(\Hent(R_i) - \Hent(R_{i-1})) \ge nh,
\]
so the corollary follows by \Pref{prop:non-uni-ent-bound}.
\end{proof}

We are ready to prove the non-unimodular case of \Tref{thm:main-graphs}.

\begin{proof}[Proof of \Tref{thm:main-graphs}, non-unimodular case]
Assume that $\Gamma$ is not unimodular. We may take $n_0\ge 4d^2\log K$. In this case, by \Cref{cor:non-uni-diffent-bound}, $\Hent(R_{2n}) - \Hent(R_n) \ge 2\log K$ for every $n\ge n_0$, so the assumption does not hold for any non-unimodular graph. Therefore the theorem is vacuously true for non-unimodular graphs.
\end{proof}

\bibliographystyle{plain}
\bibliography{refs}

\end{document}